\documentclass[a4paper]{amsart}
\usepackage{amsfonts}
\usepackage{amssymb,latexsym}

\theoremstyle{plain}
\newtheorem{theorem}{Theorem}

\newtheorem{lemma}[theorem]{Lemma}

\newtheorem{claim}[theorem]{Claim}
\theoremstyle{definition}
\newtheorem{definition}[theorem]{Definition}
\theoremstyle{remark}

\newtheorem{remark}[theorem]{Remark}

\numberwithin{equation}{section}
\numberwithin{theorem}{section}
\newenvironment{caselist}
{
\begin{list}{\bf{Case \theenumi.}}
{
 \usecounter{enumi} 
\settowidth{\itemindent}{Case 3}
\addtolength{\itemindent}{\labelsep} 
\settowidth{\labelwidth}{Case 3}
\setlength{\leftmargin}{0pt}
}
}
{\end{list}}
\newenvironment{enumeq}
{ 
\begin{enumerate} 
\setcounter{enumi}{\value{equation}}
\renewcommand{\theenumi}{(\arabic{section}.\arabic{enumi})}
}
{ 
\end{enumerate}
\setcounter{equation}{\value{enumi}}
 }
\newcommand{\range}{\operatorname{range}}

\begin{document}
\title[Minimal clones generated by majority operations]{Minimal clones
generated by majority operations}
\author[T. Waldhauser]{Tam\'{a}s Waldhauser}
\address{Bolyai Institute\\
University of Szeged\\
Aradi v\'{e}rtan\'{u}k tere 1, H6720, Szeged, Hungary}
\email{twaldha@math.u-szeged.hu}
\thanks{This research was supported by Ministry of Culture and Education of
Hungary, grant No. FKFP 0877/1997.}

\begin{abstract}
The minimal majority functions of the four-element set are determined.
\end{abstract}

\maketitle

\section{Introduction}

A set $C$ of finitary operations on a set $A$ is a \textit{clone} if it is
closed under composition of functions and contains all projections. In this
paper we shall be concerned only with clones on a finite set.

The set of all finitary operations on $A$ is a clone as well as the set of
all projections. These are the largest and the smallest clones on $A$, the
latter is often called the trivial clone.

The clone generated by a set $F$ of finitary functions on $A$ is the
intersection of all clones containing $F$, i.e. the smallest clone
containing $F$. This clone is denoted by $[F]$. If $F=\{f\}$ then we simply
write $[f]$. Clearly $[F]$ consists of those functions which can be obtained
from the elements of $F$ and from the projections by a finite number of
compositions. In other words, $[F]$ is the set of term functions of the
algebra $\langle A,F\rangle$.

We say that the $n$-ary function $f$ \textit{preserves} the relation $\rho
\subseteq A^{k}$ if for all $a_{i,j}\in A$ $\left( i=1,\ldots ,k,~j=1,\ldots
,n\right) $ 
\begin{equation*}
\left( a_{11},a_{21},\ldots ,a_{k1}\right) ,\left( a_{12},a_{22},\ldots
,a_{k2}\right) ,\ldots ,\left( a_{1n},a_{2n},\ldots ,a_{kn}\right) \in \rho
\end{equation*}%
implies 
\begin{equation*}
\left( f\left( a_{11},a_{12},\ldots ,a_{1n}\right) ,f\left(
a_{21},a_{22},\ldots ,a_{2n}\right) ,\ldots ,f\left( a_{k1},a_{k2},\ldots
,a_{kn}\right) \right) \in \rho .
\end{equation*}

Preserving a relation is inherited when composing functions:\smallskip 
\begin{equation}
\text{If }f\text{ preserves a relation }\rho \text{ and }g\in \lbrack f]%
\text{, then }g\text{ also preserves }\rho \text{.}  \label{1.1}
\end{equation}

An important special case is that of unary relations: $f$ preserves $%
B\subseteq A$ iff $B$ is closed under $f$. If $f$ preserves all subsets of $%
A $ then we say that $f$ is \textit{conservative} (cf. \cite{c5}).

A clone is \textit{minimal} if it has no proper subclones except for the
trivial one. On finite sets every clone contains a minimal one (cf. \cite{c6}%
). Obviously a clone is minimal iff it is generated by every nontrivial
member of it. (By a nontrivial function we mean a function which is not a
projection.) If $C$ is a minimal clone and $f$ is nontrivial, and of minimum
arity in $C$ then we say that $f$ is a \textit{minimal function} (cf. \cite%
{c7} p.408).

By a theorem of I. G. Rosenberg \cite{c7}, every minimal clone (on a finite
set) is generated by a nontrivial minimal function $f$, for which one of the
following holds: 
\renewcommand{\theenumi}{\arabic{enumi})}
\renewcommand{\labelenumi}{\theenumi}%

\begin{enumerate}
\item \label{unary}$f$ is unary and $f^{2}(x)\approx f(x)$ or $%
f^{p}(x)\approx x$ for some prime $p$.

\item \label{binary}$f$ is binary idempotent, i.e. $f(x,x)\approx x$.

\item \label{majority}$f$ is ternary majority, i.e. $f(x,x,y)\approx
f(x,y,x)\approx f(y,x,x)\approx x$.

\item \label{semiprojection}$f$ is a semiprojection, i.e. there exists an $i$
such that $f(x_{1},x_{2},\ldots ,x_{n})=x_{i}$ whenever the arguments are
not pairwise distinct.

\item \label{minority}$f(x,y,z)=x+y+z$ where $\langle A,+\rangle $ is a
Boolean group.
\end{enumerate}

Note that in each case $f$ cannot generate a nontrivial function which is of
lesser arity than $f$. This means that $f$ is a minimal function iff $[f]$
is a minimal clone. In cases \ref{unary} and \ref{minority} the conditions
ensure the minimality of $f$, while in the other cases they do not.

In \cite{c4} Post described all clones on a two-element set, in \cite{c1} B.
Cs\'{a}k\'{a}ny determined the minimal clones of a three-element set. For
the four-element case binary minimal clones were described by B. Szczepara
in \cite{c8}.

Conservative minimal majority and binary functions were determined on any
finite set by B. Cs\'{a}k\'{a}ny in \cite{c2}.

In this paper we prove the following description of all the minimal majority
functions of a four-element set.

\begin{theorem}
\label{thm 1.1}If $C$ is a minimal clone on a four-element set $A$ and it
contains a majority function, then $C=[f]$ where $f$ is either conservative
or $\langle A,f\rangle \cong \langle \{1,2,3,4\},M_{i}\rangle $ for some $%
i\in \{1,2,3\}$ (see the table below).%
\begin{equation*}
\begin{tabular}{cr}
\begin{tabular}{|c|ccc|}
\hline
& $M_{1}$ & $M_{2}$ & $M_{3}$ \\ \hline
$(1,2,3)$ & $4$ & $4$ & $3$ \\ 
$(2,3,1)$ & $4$ & $2$ & $3$ \\ 
$(3,1,2)$ & $4$ & $3$ & $3$ \\ \hline
$(2,1,3)$ & $4$ & $2$ & $4$ \\ 
$(1,3,2)$ & $4$ & $4$ & $4$ \\ 
$(3,2,1)$ & $4$ & $3$ & $4$ \\ \hline
$\{1,2,4\}$ & $4$ & $4$ & $4$ \\ \hline
$\{1,3,4\}$ & $4$ & $4$ & $4$ \\ \hline
$(4,2,3)$ & $4$ & $4$ & $3$ \\ 
$(2,3,4)$ & $4$ & $2$ & $3$ \\ 
$(3,4,2)$ & $4$ & $3$ & $3$ \\ \hline
$(2,4,3)$ & $4$ & $2$ & $4$ \\ 
$(4,3,2)$ & $4$ & $4$ & $4$ \\ 
$(3,2,4)$ & $4$ & $3$ & $4$ \\ \hline
\end{tabular}
& 
\parbox{6 cm}{\setlength{\parindent}{12 pt} \setlength{\leftskip}{12 pt}
\normalfont{
The middle two rows mean that if $\left\{ a,b,c\right\} =\left\{ 1,2,4\right\} $ or $\left\{ 1,3,4\right\} $, then $M_{i}(a,b,c)=4$ for $i=1,2,3$. For the triplets not listed in the table the majority rule defines the value of the functions.}}%
\end{tabular}%
\end{equation*}
\end{theorem}

\section{Majority functions on finite sets}

If $C$ is a clone which is generated by a majority function then we shall
briefly say that $C$ is a majority clone.

Let $A$ be a finite set and $f$ be a majority function on $A$. We define the
range of $f$ in the following way:%
\begin{equation*}
\range(f)=\left\{ \,f(a,b,c)\,\mid \,a,b,c\in A\text{ are pairwise
distinct}\,\right\} \text{.}
\end{equation*}

A simple induction argument shows that if $g$ is a nontrivial function in a
majority clone, then $g$ is a so-called \textit{near-unanimity function},
i.e. 
\begin{equation*}
g(y,x,x,\ldots ,x,x)\approx g(x,y,x,\ldots ,x,x)\approx \ldots \approx
g(x,x,x,\ldots ,x,y)\approx x
\end{equation*}%
(cf. (7) of \cite{c2}).

In Rosenberg's theorem $f$ cannot be a near-unanimity function except for
the majority case, so any minimal subclone of a majority clone is again a
majority clone. This means that in order to prove the minimality of a
majority clone $C$, it suffices to show that any two majority functions in $%
C $ generate each other.\ 

To show the nonminimality of a clone $[f]$ we will make use of the following
facts.

\begin{lemma}
\label{lemma 2.1}Let $f$ be a majority function on $A$.

\begin{enumeq}
\item \label{2.1}If $f$ is a minimal function and it preserves $B\subseteq A$
then $f\mid _{B}$ must be a minimal function on $B$.

\item \label{2.2}If a nontrivial $g\in \left[ f\right] $ preserves some $%
B\subseteq A$ but $f$ does not, then $\left[ f\right] $ is not minimal.

\item \label{2.3}If the range of some nontrivial $g\in \lbrack f]$ does not
contain an element which belongs to the range of $f$, then $[f]$ is not
minimal. (Cf. Corollary 1.25 of \cite{c9}.)
\end{enumeq}
\end{lemma}

\begin{proof}
~

\begin{enumerate}
\item[$\protect\ref{2.1}$] Composing functions and restricting functions
commute.

\item[$\protect\ref{2.2}$] This follows from (\ref{1.1}).

\item[$\protect\ref{2.3}$] This also follows from (\ref{1.1}) since $a\,{%
\not\in }\range(g)$ iff $g$ preserves the equivalence relation whose
blocks are $\{a\}$ and $A\backslash \{a\}$.\qedhere
\end{enumerate}
\end{proof}

We can use \ref{2.1} for three-element set $B$ because we know the minimal
majority clones for such set. These are described in \cite{c1} as follows.
If $C$ is a minimal majority clone on a three-element set $A$, then there
exists an $f\in C$ such that 
\begin{equation*}
\langle A,f\rangle \cong \langle \{1,2,3\},m_{i}\rangle \text{ for some }%
i=1,2,3\text{,}
\end{equation*}%
where $m_{1},m_{2},m_{3}$ are the following majority functions.

For $\{x_{0},x_{1},x_{2}\}=\{1,2,3\}$ we have

\begin{enumerate}
\item[ ] $m_{1}(x_{0},x_{1},x_{2})=1$

\item[ ] $m_{2}(x_{0},x_{1},x_{2})=x_{0}$

\item[ ] $m_{3}(x_{0},x_{1},x_{2})=x_{i+1}$ if $x_{i}=1$ (subscripts taken
modulo $3$).
\end{enumerate}

The clones generated by $m_{1},m_{2},m_{3}$ contain 1,3, and 8 majority
functions respectively; these are shown in the following table.

\begin{equation*}
\begin{tabular}{|c|c|ccc|cccccccc|}
\hline
& $m_{1}$ & $m_{2}$ &  &  & $m_{3}$ &  &  &  &  &  &  &  \\ \hline
$(1,2,3)$ & $1$ & $1$ & $2$ & $3$ & $2$ & $3$ & $2$ & $2$ & $3$ & $2$ & $3$
& $3$ \\ 
$(2,3,1)$ & $1$ & $2$ & $3$ & $1$ & $2$ & $2$ & $2$ & $3$ & $3$ & $3$ & $3$
& $2$ \\ 
$(3,1,2)$ & $1$ & $3$ & $1$ & $2$ & $2$ & $2$ & $3$ & $2$ & $3$ & $3$ & $2$
& $3$ \\ \hline
$(2,1,3)$ & $1$ & $2$ & $1$ & $3$ & $3$ & $3$ & $2$ & $3$ & $2$ & $2$ & $3$
& $2$ \\ 
$(1,3,2)$ & $1$ & $1$ & $3$ & $2$ & $3$ & $2$ & $3$ & $3$ & $2$ & $3$ & $2$
& $2$ \\ 
$(3,2,1)$ & $1$ & $3$ & $2$ & $1$ & $3$ & $3$ & $3$ & $2$ & $2$ & $2$ & $2$
& $3$ \\ \hline
\end{tabular}%
\end{equation*}%
\medskip

Conservative minimal majority clones are described in \cite{c2} as follows:

If $C$ is a conservative minimal majority clone then there exists an $f\in C$
such that for every three-element $B\subseteq A$ there is an $i_{B}\in
\{1,2,3\}$ such that 
\begin{equation*}
\langle B,f\mid _{B}\rangle \cong \langle \{1,2,3\},m_{i_{B}}\rangle .
\end{equation*}

Now we formulate a theorem which helps us reducing the number of functions
to be checked, when searching for minimal clones.

\begin{theorem}
\label{thm 2.2}Let $f$ be a majority function on a finite set $A$. Then
there exists a majority function $g\in \lbrack f]$ which satisfies the
following identity:%
\begin{equation}
g\bigl(\,g(x,y,z)\,,\,g(y,z,x)\,,\,g(z,x,y)\,\bigr)\approx g(x,y,z). 
\tag{$\ast $}  \label{*}
\end{equation}
\end{theorem}

\begin{proof}
We define functions $f^{(k)}\ (k\geq 1)$ in the following way:%
\begin{align*}
f^{(1)}(x,y,z)& =f(x,y,z), \\
f^{(k+1)}(x,y,z)& =f\bigl(\,f^{(k)}(x,y,z)\,,\,f^{(k)}(y,z,x)\,,%
\,f^{(k)}(z,x,y)\,\bigr).
\end{align*}

We assert that 
\begin{equation*}
f^{(k+l)}(x,y,z)\approx f^{(k)}\bigl(\,f^{(l)}(x,y,z)\,,\,f^{(l)}(y,z,x)\,,%
\,f^{(l)}(z,x,y)\,\bigr)
\end{equation*}%
for $k,l\geq 1$.

This can be proved by induction on $k$; the proof is left to the reader. Let
us define a binary operation $\ast $ on the set $D=\left\{ f^{(k)}:k\in 
\mathbb{N}\hspace{0cm}\right\} $ as follows: 
\begin{equation*}
\bigl(f^{(k)}\ast f^{(l)}\bigr)(x,y,z)=f^{(k)}\bigl(\,f^{(l)}(x,y,z)\,,%
\,f^{(l)}(y,z,x)\,,\,f^{(l)}(z,x,y)\,\bigr).
\end{equation*}

The above assertion means that the map $k\mapsto f^{(k)}$ is a homomorphism
from $\langle\mathbb{N}\hspace{0cm},+\rangle$ to $\langle D,\ast\rangle$. So
the latter is a finite semigroup, hence it has an idempotent element, say $%
f^{(k)}\ast f^{(k)}=f^{(k)}$. And this is just the desired identity for $%
g=f^{(k)}$.
\end{proof}

Now we introduce some more notation. The function $g=f^{(k)}$ which
corresponds to $f$ in the theorem will be denoted by $\widehat{f}$. We put $%
\langle abc\rangle =\{(a,b,c),(b,c,a),(c,a,b)\},$ and we will use the symbol 
$f\mid _{\langle abc\rangle }\equiv u$ to mean that $%
f(a,b,c)=f(b,c,a)=f(c,a,b)\nolinebreak =\nolinebreak u$, and $f\mid
_{\langle abc\rangle }\nolinebreak =\nolinebreak p$ \ to mean that $%
f(a,b,c)\nolinebreak =\nolinebreak a,\allowbreak f(b,c,a)\nolinebreak
=\nolinebreak b,\allowbreak f(c,a,b)\nolinebreak =\nolinebreak c$.

The following lemma tells us what identity (\ref{*}) means for a majority
function.

\begin{lemma}
\label{lemma 2.3}Let $f$ be a majority function satisfying \eqref{*} and let 
$a,b,c$ be pairwise distinct elements of $A$. Let $u=f(a,b,c),\ v=f(b,c,a),\
w=f(c,a,b)$. Then $\left\vert \left\{ u,v,w\right\} \right\vert \neq 2$ and
if $u,v,w$ are pairwise different, then $f\mid _{\langle uvw\rangle }=p$.
\end{lemma}

\begin{proof}
To prove the first statement, let us suppose (without loss of generality)
that $u=v\neq w$. Then (\ref{*}) for $x=c,\ y=a,\ z=b$ yields that $%
f(w,u,v)=w$, contradicting the majority property of $f$.

The second statement of the lemma follows similarly from (\ref{*}).
\end{proof}

We can say a bit more then Lemma \ref{lemma 2.3} when $f$ is a minimal
function.

\begin{theorem}
\label{thm 2.4}If $f$ is a minimal majority function satisfying \eqref{*}
and $u=f(a,b,c),\ $\linebreak $v=f(b,c,a),\ w=f(c,a,b)$ are pairwise
different then $f${$\mid _{\langle uvw\rangle }$}$=p$\ and also {$f\mid
_{\langle vuw\rangle }$}$=p$.
\end{theorem}

\begin{proof}
By the previous lemma we have {$f\mid _{\langle uvw\rangle }$}$=p$. Now the
nontrivial superposition $g(x,y,z)=f(f(x,y,z),f(x,z,y),x)$ preserves $%
\{u,v,w\}$ hence $f$ does too, and then from the description of the minimal
majority functions on the three-element set we get the conclusion of the
theorem.
\end{proof}

\section{The four-element case}

We have seen that every conservative minimal majority clone is generated by
a function $f$ having the following property:%
\begin{equation}
f\mid _{\langle abc\rangle }\equiv u\text{ or }f\mid _{\langle abc\rangle }=p
\label{3.1}
\end{equation}%
for every $a,b,c\in A$ with a suitable $u$ (depending of course, on $a,b,c$).

One would hope that it holds for nonconservative clones too. In the first
part of this section we are going to try to prove this for a four-element $A$%
. It will turn out, that the conjecture is not true, but (in the
four-element case) there is essentially only one exception. In the second
part we determine the minimal ones among the functions satisfying property (%
\ref{3.1}), and in the third part we prove the minimality of the clones we
have found.

\subsection*{3.1}

$~\smallskip $

Let $S$ denote the set of those majority functions on the set $A=\{1,2,3,4\}$
for which (\ref{3.1}) holds for any $a,b,c\in A$.

In this section we will show that a minimal majority function which
satisfies $(\ast )$ must belong to the set $S$, or it is isomorphic to $%
M_{2} $. Since we will consider the values of the functions on the set $%
\{1,2,3\}$, we introduce one more notation. Let $\left[ p,q,r;s,t,u\right] $
denote the set of majority functions $f$ on $A$ for which $%
f(1,2,3)\nolinebreak =\nolinebreak p,\ \allowbreak f(2,3,1)\nolinebreak
=\nolinebreak q,\ \allowbreak f(3,1,2)\nolinebreak =\nolinebreak r,\
\allowbreak f(2,1,3)\nolinebreak =\nolinebreak s,\allowbreak \
f(1,3,2)\nolinebreak =\nolinebreak t,\ \allowbreak f(3,2,1)\nolinebreak
=\nolinebreak u$. If we do not want to specify all these six values of $f$,
than we will use $\ast $ to indicate an arbitrary element of $A$. For
example $f\in \left[ 4,\ast ,\ast ;\ast ,\ast ,\ast \right] $ means just
that $f(1,2,3)\nolinebreak =\nolinebreak 4$. The letters $a,b,c,d$ will
always denote arbitrary distinct elements of $A$, i.e. $\{1,2,3,4\}=A=%
\{a,b,c,d\}$.

First we define and examine a superposition which we will use frequently
later on. For a ternary function $f$ let $f_{x}$, $f_{y}$, $f_{z}$ stand for
the composite functions where the first, second resp. third variable of $f$
is replaced by $f$ itself.%
\begin{align*}
f_{x}(x,y,z)& =f(f(x,y,z),y,z) \\
f_{y}(x,y,z)& =f(x,f(x,y,z),z) \\
f_{z}(x,y,z)& =f(x,y,f(x,y,z))
\end{align*}

We will briefly write $f_{zy}$\ instead of $(f_{z})_{y}$. We will also use
the convention that lower indices have priority to upper ones. So $%
f_{zy}^{(2)}$ \ means $(f_{zy})^{(2)}$ \ and not $(f^{(2)})_{zy}$, and also $%
\widehat{f}_{zy}$ \ stands for $\widehat{(f_{zy})}$.

The proof of the following lemma is just a straightforward calculation, so
we omit it.

\begin{lemma}
\label{lemma 3.1}If $f(a,b,c)\neq d$ then $f_{zy}(a,b,c)=f(a,b,c)$. If this
is not the case, then $f_{zy}(a,b,c)=f(a,b,d)$ if the latter does not equal $%
d$. If it does, then $f_{zy}(a,b,c)=f(a,d,c)$ if it is not $b$. If $%
f(a,d,c)=b$ then $f_{zy}(a,b,c)=f(a,d,b)$.
\end{lemma}

From now on $f$ will always denote an arbitrary majority function on $A$,
satisfying (\ref{*}). In the following lemma we prove a nice property of $f$%
, then through five claims we reach the main result of this section, which
is stated in Theorem \ref{thm 3.8}. Let us recall that $\langle abc\rangle $
is just the set $\{(a,b,c),\,(b,c,a),\,(c,a,b)\}$, hence $f(\langle
abc\rangle )$ denotes $\{f(a,b,c),\,f(b,c,a),\,f(c,a,b)\}$.

\begin{lemma}
\label{lemma 3.2}If $f$ is minimal and $f(\langle abc\rangle)\subseteq
\{a,b,c\}$ then either {$f\mid_{\langle abc\rangle}$}$=p$\ and {$f\mid
_{\langle bac\rangle}$}$=p$\ or {$f\mid_{\langle abc\rangle}$}$\equiv u$\
and {$f\mid_{\langle bac\rangle}$}$\equiv v$\ for some $u,v\in A$.
\end{lemma}

\begin{proof}
The set $f(\langle abc\rangle )$ has three or one elements by Lemma \ref%
{lemma 2.3}. If it has three elements then it is $\{a,b,c\}$, and then by
Theorem \ref{thm 2.4} we have {$f\mid _{\langle abc\rangle }$}$=p$\ and {$%
f\mid _{\langle bac\rangle }$}$=p$. In the latter case we may suppose {$%
f\mid _{\langle abc\rangle }$}$\equiv a$. If $d\notin f(\langle bac\rangle )$
then $f$ preserves $\{a,b,c\}$ and then the description of the minimal
majority functions on the three-element set yields {$f\mid _{\langle
bac\rangle }$}$\equiv v$. If $a\in f(\langle bac\rangle )$ then we permute
cyclically the variables to have $f(b,a,c)=a$, and then $g^{(2)}$ preserves $%
\{a,b,c\}$ for the superposition $g$ of Theorem \ref{thm 2.4}, contradicting
the minimality of $f$. Finally, if $a\notin f(\langle bac\rangle )$ but $%
d\in f(\langle bac\rangle )$ then $f(\langle bac\rangle )=\{b,c,d\}$. Now we
may suppose $f(b,a,c)\nolinebreak =\nolinebreak c,\allowbreak \
f(a,c,b)\nolinebreak =\nolinebreak d,\ \allowbreak f(c,b,a)\nolinebreak
=\nolinebreak b$ or $f(b,a,c)\nolinebreak =\nolinebreak b,\ \allowbreak
f(a,c,b)\nolinebreak =\nolinebreak d,\ \allowbreak f(c,b,a)\nolinebreak
=\nolinebreak c$ after a cyclic permutation of variables. In the first case $%
g$, in the second case $g^{(2)}$ shows that $f$ is not minimal, since they
preserve $\{a,b,c\}$.
\end{proof}

\renewcommand{\theenumi}{(\arabic{enumi})}
\renewcommand{\labelenumi}{\theenumi}%

\begin{claim}
\label{claim 3.3}In either of the following four cases $f$ is not minimal.

\begin{enumerate}
\item \label{claim 3.3 case 1}$f\in \left[ 4,2,1;\ast ,\ast ,\ast \right] $

\item \label{claim 3.3 case 2}$f\in \left[ 4,1,2;\ast ,\ast ,\ast \right] $

\item $f\in \left[ 4,1,3;\ast ,\ast ,\ast \right] $

\item $f\in \left[ 4,3,1;\ast ,\ast ,\ast \right] $
\end{enumerate}
\end{claim}

\begin{proof}
~

\begin{enumerate}
\item Lemma \ref{lemma 3.1} shows that $f_{zy}$\ preserves $\{1,2,3\}$ (and
hence $f$ is not minimal) except when $f(3,2,1)\nolinebreak =\nolinebreak
4,\ \allowbreak f(3,2,4)\nolinebreak =\nolinebreak 4$ and $%
f(3,4,1)\nolinebreak =\nolinebreak 4$ or $f(3,4,1)\nolinebreak =\nolinebreak
2,\allowbreak f(3,4,2)\nolinebreak =\nolinebreak 4$. Let us examine the set $%
f(\langle 213\rangle )$. It has one or three elements by Lemma \ref{lemma
2.3}. If it is $\{1\}$, $\{2\}$ or $\{3\}$, then Lemma \ref{lemma 3.2} shows
that $f$ cannot be minimal. If we have {$f\mid _{\langle 213\rangle }$}$%
\equiv 4$, then we can compute that $f_{zy}\in \left[ 1,2,1;2,u,4\right] $,
where $u\neq 4$. Depending on whether $u=2,3,1$ resp. it can be shown that $%
\widehat{f}_{zy}$\ or $f_{zy}(y,z,f_{zy}(x,y,z))$ preserves $\{1,2,3\}$ or $%
\widehat{f}_{zy}$\ is not minimal by Lemma \ref{lemma 3.2}. Now let us
suppose that $f(\langle 213\rangle )$ is a three-element set. If it is $%
\{1,3,4\}$, then Theorem \ref{thm 2.4} implies {$f\mid _{\langle 341\rangle
} $}$=p$, hence $f(3,4,1)\nolinebreak =\nolinebreak 3$, but we have seen
that it is $4$ or $2$. Similarly $f(\langle 213\rangle )=\{1,2,3\},\{2,3,4\}$
is also impossible. So $f(\langle 213\rangle )$ can be nothing else but $%
\{1,2,4\}$. Since $f(3,2,1)=4$ there are only two possibilities: $f\in \left[
4,2,1;1,2,4\right] $ or $f\in \left[ 4,2,1;2,1,4\right] $, and then $%
f_{zy}\in \left[ 1,2,1;1,2,4\right] $ or $f_{zy}\in \left[ 1,2,1;2,1,4\right]
$. In both cases Lemma \ref{lemma 3.2} yields that $\widehat{f}_{zy}$\ is
not minimal, hence neither is $f$, and we have finished the proof.

\item Here we can use the same argument, the only difference is that in this
case $f_{zy}\in \left[ 1,1,2;\ast ,\ast ,\ast \right] $.

\item The function $f(x,z,y)$ is isomorphic to a function which is not
minimal by case \ref{claim 3.3 case 1}. (We shall note here that changing
the second and third variable does not influence the identity (\ref{*}).)

\item Now $f(x,z,y)$ falls under case \ref{claim 3.3 case 2} after renaming
the elements of the base set.\qedhere
\end{enumerate}
\end{proof}

\begin{claim}
\label{claim 3.4}If $f\in \left[ 4,3,2;\ast ,\ast ,\ast \right] $ then $f$
is not minimal.
\end{claim}

\begin{proof}
Just as in the previous claim, we examine $f(\langle 213\rangle )$. If it is 
$\{1\}$, $\{2\}$, $\{3\}$ or $\{1,2,3\}$ then Lemma \ref{lemma 3.2} shows
that $f$ is not minimal. If {$f\mid _{\langle 213\rangle }$}$\equiv 4$ then $%
g(x,y,z)=f(z,y,f(x,y,z))\in \left[ 3,3,2;u,2,v\right] $. If none of $u$ and $%
v$ equals $4$, then $g$ preserves $\{1,2,3\}$. If $u\neq 3$ then $\widehat{h}
$ does, for $h(x,y,z)=g(g(x,y,z),z,x)$. Only $u\nolinebreak =\nolinebreak
3,\allowbreak v\nolinebreak =\nolinebreak 4$ remains, but in this case $%
\widehat{g}\in \left[ 3,3,3;2,4,3\right] $, and it is not a minimal function
by Lemma \ref{lemma 3.2}. Now let us suppose that $f(\langle 213\rangle )$
is a three-element set containing $4$. If it is $\{1,2,4\}$, then according
to Claim \ref{claim 3.3}, we must have $f\in \left[ 4,3,2;2,1,4\right] $ or $%
f\in \left[ 4,3,2;1,2,4\right] $, and in both cases $f_{zy}$\ preserves $%
\{1,2,3\}$. Similarly $f(\langle 213\rangle )\nolinebreak =\nolinebreak
\{1,3,4\}$ implies $f\in \left[ 4,3,2;4,1,3\right] $ or $f\in \left[
4,3,2;4,3,1\right] $, and again $f_{zy}$\ preserves $\{1,2,3\}$. Finally, if 
$f(\langle 213\rangle )=\{2,3,4\}$ then $f\in \left[ 4,3,2;3,4,2\right] $ or 
$f\in \left[ 4,3,2;2,4,3\right] $. In the first case $%
g(x,y,z)=f(z,y,f(x,y,z))$ preserves $\{1,2,3\}$, in the second case $%
\widehat{g}$ does.
\end{proof}

\begin{claim}
\label{claim 3.5}If $f\in\left[ 4,2,3;2,1,4\right] $ or $f\in\left[
4,2,3;4,1,3\right] $ then $f$ is not minimal.
\end{claim}

\begin{proof}
In the first case $f_{z}$ preserves $\{1,2,3\}$, in the second case $f_{y}$
does.
\end{proof}

\begin{claim}
\label{claim 3.6}If $f\in\left[ 4,2,3;2,4,3\right] $ then $f=M_{2}$.
\end{claim}

\begin{proof}
Lemma \ref{lemma 2.3} yields {$f\mid_{\langle234\rangle}$}$=p$ \ and {$%
f\mid_{\langle324\rangle}$}$=p$. Let us put $g(x,y,z)=f(x,y,f(x,y,z))$. Then 
$g\in\left[ f(1,2,4),2,3;2,f(1,3,4),3\right] $. If none of $f(1,2,4)$, $%
f(1,3,4)$ equals $4$ then $g$ preserves $\{1,2,3\}$. If one of them equals $%
4 $, the other not, then $\widehat{g}$ is not minimal by Lemma \ref{lemma
3.2}. So we must have $f(1,2,4)=f(1,3,4)=4$. In the same way we get $%
f(2,1,4)=f(3,1,4)=4$, $f(1,4,2)=f(1,4,3)=4$, etc. by using $%
g(x,y,z)=f(y,x,f(x,y,z))$, $f(x,f(x,y,z),y)$ etc.
\end{proof}

\begin{claim}
\label{claim 3.7}If $f\in\left[ 4,2,3;4,4,4\right] $ then $f$ is not minimal.
\end{claim}

\begin{proof}
If $f_{zy}(2,1,3)=1$ then for $h(x,y,z)=f_{zy}(z,x,f_{zy}(x,y,z))$ either $%
\widehat{h}$ preserves $\{1,2,3\}$ or fails to be minimal by Lemma \ref%
{lemma 3.2}. If $f_{zy}(2,1,3)\neq1$ then the same holds for $f_{zy}$\
itself, except when $f_{zy}\in\left[ 4,2,3;2,4,3\right] $. In this case
Claim \ref{claim 3.6} yields $f_{zy}=M_{2}$. We will see later that the
clone generated by $M_{2}$ contains only three majority functions, and none
of them equals $f$.
\end{proof}

\begin{theorem}
\label{thm 3.8}Any minimal nonconservative majority function on $A$ which
satisfies \eqref{*} is isomorphic to $M_{2}$ or it belongs to the set $S$.
\end{theorem}

\begin{proof}
Let $f$ be a function as stated in the theorem. According to Claim \ref%
{claim 3.3} and Claim \ref{claim 3.4}, for every $a,b,c$ if neither {$f\mid
_{\langle abc\rangle }$}$=p$ \ nor {$f\mid _{\langle abc\rangle }$}$\equiv u$
holds, then we must have that on two of the three triplets of $\langle
abc\rangle $ the value of $f$ equals the first variable, while on the third
one $f$ equals $d$. If $f\notin S$ then this case really appears, so we can
suppose (after an isomorphism if necessary) that $f(1,2,3)\nolinebreak
=\nolinebreak 4,\allowbreak f(2,3,1)\nolinebreak =\nolinebreak 2,\allowbreak
f(3,1,2)\nolinebreak =\nolinebreak 3$. Now if $4\notin f(\langle 213\rangle
) $ then we get a contradiction by Lemma \ref{lemma 3.2}. If {$f\mid
_{\langle 213\rangle }$}$\equiv 4$ than Claim \ref{claim 3.7} implies that $%
f $ is not minimal. So $f(\langle 213\rangle )$ must be a three-element set
containing $4$, and then again by Claims \ref{claim 3.3} and \ref{claim 3.4}
we must have $f\in \left[ 4,2,3;2,1,4\right] $, $f\in \left[ 4,2,3;4,1,3%
\right] $, or $f\in \left[ 4,2,3;2,4,3\right] $. The first two of these is
impossible by Claim \ref{claim 3.5}, and in the third case Claim \ref{claim
3.6} shows that $f$ equals $M_{2}$.
\end{proof}

\subsection*{3.2}

$~\smallskip $

In this section we are going to search for the minimal functions of the set $%
S$. The conservative ones are already described, so we deal only with
nonconservative functions. We assume $f$ to be such a function and we will
prove several properties of $f$, until we find that only a few functions
(essentially two) possess these properties, and these happen to be minimal.

\begin{definition}
\label{def 3.9}A ternary function $g$ is said to be cyclically commutative
if it is invariant under the cyclic permutation of variables, i.e. 
\begin{equation*}
g(x,y,z)\approx g(y,z,x)\approx g(z,x,y).
\end{equation*}
\end{definition}

\begin{claim}
\label{claim 3.10}The function $f$ is cyclically commutative.
\end{claim}

\begin{proof}
For contradiction, let us suppose {$f\mid _{\langle 124\rangle }$}$=p$, and
then by Lemma \ref{lemma 3.2} we have also {$f\mid _{\langle 214\rangle }$}$%
=p$. Since $f$ is not conservative, we may also suppose (without loss of
generality) that {$f\mid _{\langle 123\rangle }$}$\equiv 4$, and again by
Lemma \ref{lemma 3.2} we must have {$f\mid _{\langle 213\rangle }$}$\equiv u$%
. First let us suppose $u\neq 4$. Then $f_{zy}$\ preserves $\{1,2,3\}$
except when $f(3,1,2)\nolinebreak =\nolinebreak f(3,1,4)\nolinebreak
=\nolinebreak f(3,4,2)\nolinebreak =\nolinebreak 4$ or $f(3,1,2)\nolinebreak
=\nolinebreak f(3,1,4)\nolinebreak =\nolinebreak f(3,4,1)\nolinebreak
=\nolinebreak 4$ and $f(3,4,2)\nolinebreak =\nolinebreak 1$. In the first
case $\widehat{f}_{zy}\in \left[ 1,2,4;u,u,u\right] $ so it is not a minimal
function by Lemma \ref{lemma 3.2}, while in the second case $f_{zy}\in \left[
1,1,4;u,u,u\right] $, hence $\widehat{f}_{zy}$\ preserves $\{1,2,3\}$. If $%
u=4$ then we have also $f_{zy}\in \left[ 1,2,4;\ast ,\ast ,\ast \right] $ or 
$f_{zy}\in \left[ 1,1,4;\ast ,\ast ,\ast \right] $ or $f_{zy}(\langle
123\rangle )\subseteq \{1,2,3\}$. The first case is impossible, since then
Theorem \ref{thm 3.8} implies that $\widehat{f}_{zy}$\ is isomorphic to $%
M_{2}$, but then $f\notin \lbrack \widehat{f}_{zy}]$ shows that $f$ is not
minimal. (In fact, the clone generated by $M_{2}$ contains no function from $%
S$ except for the first projection.) For $\langle 213\rangle $ we have also
three possibilities: $f_{zy}\in \left[ \ast ,\ast ,\ast ;2,1,4\right] $, $%
f_{zy}\in \left[ \ast ,\ast ,\ast ;2,2,4\right] $ and $f_{zy}(\langle
213\rangle )\subseteq \{1,2,3\}$. The first one of these is impossible for
the same reason as above. In the remaining cases $\widehat{f}_{zy}$\
preserves $\{1,2,3\}$.
\end{proof}

In the following we suppose $f$ to be a nonconservative cyclically
commutative minimal majority function on $A$. In \cite{c3} these are
determined by computer, here we give a straightforward description. We again
suppose {$f\mid_{\langle123\rangle}$}$\equiv4$, and {$f\mid_{\langle213%
\rangle}$}$\equiv u$. First we show that $f$ preserves $\{1,2,4\}$, $%
\{1,3,4\}$ and $\{2,3,4\}$.

\begin{claim}
\label{claim 3.11}The only subset of $A$ not preserved by $f$ is $\{1,2,3\}$.
\end{claim}

\begin{proof}
We separate two cases upon $u$.%
\renewcommand{\theenumi}{\arabic{enumi}}
\renewcommand{\labelenumi}{\theenumi}%

\begin{caselist}
\item \label{claim 3.11 case 1}{$f\mid _{\langle 213\rangle }$}$\equiv u$ $%
\neq 4$. For contradiction let us suppose that $f$ does not preserve $%
\{1,2,4\}$. Then {$f\mid _{\langle 214\rangle }$}$\equiv 3$ \ or {$f\mid
_{\langle 124\rangle }$}$\equiv 3$. In the first case $f(y,x,f(x,y,z)^{(2)}$
preserves $\{1,2,3\}$ or $\{1,2,4\}$, in the second case $%
f(x,f_{z}(x,y,z),z)^{(2)}$ or $f_{z}$ preserves $\{1,2,3\}$ depending on
whether $4\in \{f(2,3,4),f(3,1,4)\}$ or not.

\item {$f\mid _{\langle 213\rangle }$}$\equiv 4$. What we have already
proved of this claim means that if {$f\mid _{\langle abc\rangle }$}$\equiv d$
and {$f\mid _{\langle bac\rangle }$}$\equiv a$ then $f$ preserves the other
three subsets of $A$, namely $\{a,b,d\}$, $\{a,c,d\}$, $\{b,c,d\}$. So if we
again suppose that $f$ does not preserve $\{1,2,4\}$ then we must have {$%
f\mid _{\langle 124\rangle }$}$\equiv 3$ and {$f\mid _{\langle 214\rangle }$}%
$\equiv 3$. Similarly, {$f\mid _{\langle 234\rangle }$}$\equiv 1$ if and
only if {$f\mid _{\langle 324\rangle }$}$\equiv 1$ \ and {$f\mid _{\langle
134\rangle }$}$\equiv 2$ iff {$f\mid _{\langle 314\rangle }$}$\equiv 2$. One
can check that $f_{z}$ or $\widehat{f_{z}}$ preserves $\{1,2,3\}$ or $%
\{1,2,4\}$ except for only two functions (up to isomorphism and permutation
of variables). For both of them {$f\mid _{\langle 134\rangle }$}$\equiv 3$
and {$f\mid _{\langle 314\rangle }$}$\equiv 4$, and for one we have {$f\mid
_{\langle 234\rangle }$}$\equiv 1$ and {$f\mid _{\langle 324\rangle }$}$%
\equiv 1$, for the other one {$f\mid _{\langle 234\rangle }$}$\equiv 4$ and {%
$f\mid _{\langle 324\rangle }$}$\equiv 3$. In both cases $\widehat{f_{z}}\in %
\left[ 4,4,4;3,3,3\right] $, hence by Case \ref{claim 3.11 case 1} it
preserves $\{1,2,4\}$. We supposed that $f$ does not preserve this set, so $%
f\notin \lbrack \widehat{f}_{z}]$ and this contradicts the minimality of $f$.%
\qedhere
\end{caselist}
\end{proof}

We have proved that if $f\in S$ is a minimal function, then $f$ is
cyclically commutative and preserves all but one three-element subsets of $A$%
. In the following two claims -- as usually -- we suppose that $f$ preserves 
$\{1,2,4\}$, $\{1,3,4\}$, $\{2,3,4\}$ and {$f\mid _{\langle 123\rangle }$}$%
\equiv 4$, {$f\mid _{\langle 213\rangle }$}$\equiv u$. Depending on whether $%
u=4$ or not, we will finally reach $M_{1}$ and $M_{3}$.

\begin{claim}
\label{claim 3.12}If {$f\mid_{\langle213\rangle}$}$\equiv u$ $\neq4$ then $%
\langle A;f\rangle\cong\langle A;M_{3}\rangle$.
\end{claim}

\begin{proof}
We can suppose {$f\mid _{\langle 213\rangle }$}$\equiv 3$ without loss of
generality. We also suppose {$f\mid _{\langle 124\rangle }$}$\equiv 4$,\ {$%
f\mid _{\langle 314\rangle }$}$\equiv 4$,\ {$f\mid _{\langle 234\rangle }$}$%
\equiv 4$, since otherwise $\widehat{f}_{zy}$ preserves $\{1,2,3\}$. For $%
g(x,y,z)=f(y,x,f(x,y,z))$ we have $g\in \left[
f(1,2,4),f(3,2,4),f(1,3,4);4,3,3\right] $. If none of $f(1,2,4)$, $f(3,2,4)$%
, $f(1,3,4)$ equals $4$, then $g^{(3)}$ preserves $\{1,2,3\}$. If there is a 
$4$ amongst them, but $3$ does not appear, then we put $%
h(x,y,z)=f(g(x,y,z),g(z,y,x),g(x,z,y))$ and one calculate that the range of $%
h^{(2)}$ does not contain $4$, hence $f$ is not minimal by \ref{2.3}. Only
nine functions remain; for two of them $g$ is isomorphic to $M_{3}$, hence $%
f $ is also. (The clone generated by $M_{3}$ contains only two functions
from $S$, and only one of them can be equal to $f$.) If $g\in \left[
2,4,3;4,4,3\right] $ then $g(y,g(y,z,x),g(x,y,z))^{(2)}$ preserves $%
\{1,2,3\} $, if $g\in \left[ 1,3,4;4,4,3\right] $ then $%
g(g(x,y,z),y,g(y,z,x))^{(2)}$ does so. In the remaining five cases $%
\{1,2,3\} $ is preserved by $f(g(x,y,z),g(z,x,y),g(y,z,x))$.
\end{proof}

\begin{claim}
\label{claim 3.13}If {$f\mid_{\langle213\rangle}$}$\equiv4$ \ then $\langle
A;f\rangle\cong\langle A;M_{1}\rangle$.
\end{claim}

\begin{proof}
Let $U=\{f(1,2,4),f(3,1,4),f(2,3,4)\}$ and $V=\{f(2,1,4),f(1,3,4),\linebreak
f(3,2,4)\}$. If $U\neq \{4\}$ and $V\neq \{4\}$, then $\widehat{f}_{zy}$\
preserves $\{1,2,3\}$. If $U=V=\{4\}$ then $f=M_{1}$. Now let us suppose $%
U=\{4\}\neq V$. If $4\notin V$ then $\widehat{f}_{zy}$\ is not minimal by
Lemma \ref{lemma 3.2}. If this is not the case then by the previous claim $%
\widehat{f}_{zy}$\ is isomorphic to $M_{3}$, but the clone $[M_{3}]$
contains no function which isomorphic to $f$. The case $V=\{4\}\neq U$ is
similar.
\end{proof}

\subsection*{3.3}

$~\smallskip $

We have now -- up to isomorphism -- only three functions: $M_{1},M_{2},M_{3}$%
, and these generate minimal clones.

\begin{theorem}
\label{thm 3.14}$M_{1},M_{2},M_{3}$ are minimal functions on $\{1,2,3,4\}$.
\end{theorem}

\begin{proof}
The proof is the same for all the three functions, so let $f$ be any of
them. This function preserves the equivalence relation whose blocks are $%
\{1,4\},\{2\},\{3\}$, and its range does not contain the element $1$.
According to (\ref{1.1}) and \ref{2.3}, the same is valid for an arbitrary
majority function $g$ in $[f]$. These properties determine {$g\mid
_{\{1,2,3\}}$} provided {$g\mid _{\{2,3,4\}}$} is given. Since $f$ preserves 
$\{2,3,4\}$ and $f${$\mid _{\{2,3,4\}}$} is minimal, there exists an $h\in
\lbrack g]$ such that {$h\mid _{\{2,3,4\}}$}$=${$f\mid _{\{2,3,4\}}$}. Now $%
h $ has also the above mentioned two properties, so {$h\mid _{\{1,2,3\}}$}
is uniquely determined: it can be nothing else than {$f\mid _{\{1,2,3\}}$}.
On $\{1,2,4\}$ and on $\{1,3,4\}$ $f$ is constant $4$, consequently so are $%
g $ and $h$, hence $h=f$. Thus, for arbitrary $g\in \lbrack f]$, $f\in
\lbrack g] $ also holds, proving that $[f]$ is a minimal clone.
\end{proof}

\begin{remark}
From the proof it is clear, that restriction to $\{2,3,4\}$ gives a
one-to-one correspondence between the majority functions in $[M_{i}]$ and $%
[m_{i}]$. Hence these clones contain also 1,3 or 8 majority functions. They
can be seen in the following table.
\end{remark}

\begin{equation*}
\bigskip 
\begin{tabular}{|c|c|ccc|cccccccc|}
\hline
& $M_{1}$ & $M_{2}$ &  &  & $M_{3}$ &  &  &  &  &  &  &  \\ \hline
$(1,2,3)$ & $4$ & $4$ & $2$ & $3$ & $3$ & $3$ & $4$ & $3$ & $4$ & $4$ & $3$
& $4$ \\ 
$(2,3,1)$ & $4$ & $2$ & $3$ & $4$ & $3$ & $4$ & $3$ & $3$ & $4$ & $3$ & $4$
& $4$ \\ 
$(3,1,2)$ & $4$ & $3$ & $4$ & $2$ & $3$ & $3$ & $3$ & $4$ & $4$ & $4$ & $4$
& $3$ \\ \hline
$(2,1,3)$ & $4$ & $2$ & $4$ & $3$ & $4$ & $3$ & $4$ & $4$ & $3$ & $4$ & $3$
& $3$ \\ 
$(1,3,2)$ & $4$ & $4$ & $3$ & $2$ & $4$ & $4$ & $4$ & $3$ & $3$ & $3$ & $3$
& $4$ \\ 
$(3,2,1)$ & $4$ & $3$ & $2$ & $4$ & $4$ & $4$ & $3$ & $4$ & $3$ & $3$ & $4$
& $3$ \\ \hline
$\{1,2,4\}$ & $4$ & $4$ & $4$ & $4$ & $4$ & $4$ & $4$ & $4$ & $4$ & $4$ & $4$
& $4$ \\ \hline
$\{1,3,4\}$ & $4$ & $4$ & $4$ & $4$ & $4$ & $4$ & $4$ & $4$ & $4$ & $4$ & $4$
& $4$ \\ \hline
$(4,2,3)$ & $4$ & $4$ & $2$ & $3$ & $3$ & $3$ & $4$ & $3$ & $4$ & $4$ & $3$
& $4$ \\ 
$(2,3,4)$ & $4$ & $2$ & $3$ & $4$ & $3$ & $4$ & $3$ & $3$ & $4$ & $3$ & $4$
& $4$ \\ 
$(3,4,2)$ & $4$ & $3$ & $4$ & $2$ & $3$ & $3$ & $3$ & $4$ & $4$ & $4$ & $4$
& $3$ \\ \hline
$(2,4,3)$ & $4$ & $2$ & $4$ & $3$ & $4$ & $3$ & $4$ & $4$ & $3$ & $4$ & $3$
& $3$ \\ 
$(4,3,2)$ & $4$ & $4$ & $3$ & $2$ & $4$ & $4$ & $4$ & $3$ & $3$ & $3$ & $3$
& $4$ \\ 
$(3,2,4)$ & $4$ & $3$ & $2$ & $4$ & $4$ & $4$ & $3$ & $4$ & $3$ & $3$ & $4$
& $3$ \\ \hline
\end{tabular}%
\end{equation*}

\end{document}